\documentclass{article}
\usepackage{mtgrimes}
\usepackage{lineno} 

\newif\ifdetailed
\detailedfalse 

\title{Compactifications of universal moduli spaces of vector bundles and the log-minimal model program on $\overline{M}_g$ 
\ifdetailed
\\ \vspace{2mm} {\large Detailed Version}
\fi
}
\author{Matthew Grimes}

\begin{document}
\maketitle
\begin{abstract}
Recent work on the log-minimal model program for the moduli space of curves, as well as past results of Caporaso, Pandharipande, and Simpson motivate an investigation of compactifications of the universal moduli space of slope semi-stable vector bundles over moduli spaces of curves arising in the Hassett--Keel program.
Our main result is the construction of a compactification of the universal moduli space of vector bundles over several of these moduli spaces, along with a complete description in the case of pseudo-stable curves.
\end{abstract}
\ifdetailed
\else
\fi
\section{Introduction}
Moduli spaces of vector bundles over smooth curves have long been a subject of interest in algebraic geometry.
Recall that due to the work of Mumford, Newstead, and Seshadri there is a projective moduli space $U_{e,r}(C)$ of slope semi-stable vector bundles of degree $e$ and rank $r$ over any fixed complex projective curve $C$.
The result has since been generalized to other settings by many others.
In particular, in \cite{simpson1994moduli}, Simpson constructed a relative moduli space of slope semi-stable sheaves on families of polarized projective schemes.

As a consequence of Simpson's result, there is a relative moduli space of slope semi-stable vector bundles for the universal curve $C^\circ_g\rightarrow M^\circ_g$ over the moduli space of smooth, automorphism-free curves, polarized by the relative canonical bundle (for $g\geq 2$). 
Though the coarse moduli space of Deligne--Mumford stable curves, $\overline{M}_g$, does not admit a universal curve, it is natural to ask whether there is a moduli space of slope semi-stable bundles over $\overline{M}_g$ that compactifies the universal moduli space over the space of automorphism-free curves.
Caporaso and Pandharipande affirmatively answered this question in the case $r=1$ and for general rank respectively in \cite{caporaso1994compactification,pandharipande1996}.

Precisely, in \cite{pandharipande1996}, Pandharipande constructed a relative moduli space for $g\geq 2$
\[\overline{U}_{e,r,g}\rightarrow\overline{M}_g\]
parametrizing slope semi-stable torsion-free sheaves of uniform rank $r$ and degree $e$, with a dense open subset that can be identified with the uniform rank locus in Simpson's moduli space.

More recently, with the aim of providing a modular interpretation for the canonical model of the moduli space of curves, there has been interest in understanding alternate modular compactifications of the moduli space of curves.
The Hassett--Keel program outlines a principle for applying the log-minimal model program to the moduli space of genus $g$ curves to obtain a modular interpretation of the canonical model by studying spaces of the form
\[\overline{M_g}(\alpha):=\proj\left( \bigoplus_nH^0(\overline{M_g},n(K_{\overline{M}_g}+\alpha\Delta)) \right),\]
where $\Delta$ is the boundary divisor in $\overline{M}_g$ and $\alpha\in[0,1]\cap\Q$.
One has $\overline{M}_g(1)=\overline{M}_g$ and $\overline{M}_g(0)$ equal to the canonical model of $M_g$ for $g\gg 0$ (\cite{harris1982kodaira,eisenbud1987kodaira,farkas2009global}).   
We direct the reader as well to \cite{fedorchuk2010alternate,hyeon2010outline} for more details.

The first steps in the program have been worked out in \cite{hassett-hyeon2009,hassett2013log,afsflip-existence,afsflip-loc,afsflip-proj}.
In particular, Hassett and Hyeon showed (\cite{hassett-hyeon2009}) that the first birational modification occurs at $\alpha=9/11$ and is a divisorial contraction to the space
\[\overline{M}_g(9/11)\cong\overline{M}^{ps}_g,\]
where $\overline{M}^{ps}_g$ is Schubert's moduli space of pseudo-stable curves (see \cite{schubert1991new}).
Recall that $\overline{M}^{ps}_g$ is a coarse moduli space for the moduli functor of pseudo-stable curves (see Section \ref{section:moduli-problem} for a precise definition).
The contraction essentially replaces elliptic tails with cusps.
Recall that, as seen in \cite{hassett2005classical}, the moduli space $\overline{M}_2^{ps}$ contains a GIT semi-stable point.

Given $\alpha$ and a modular interpretation of $\overline M_g(\alpha)$, it is natural to ask if there exists a compactified universal moduli space of slope semi-stable sheaves
\[\overline{U}_{e,r,g}(\alpha)\rightarrow\overline{M}_{g}(\alpha)\]
which co-represents the moduli stack $\overline{\mathcal{U}}_{e,r,g}(\alpha)$ of slope semi-stable sheaves on $\alpha$-stable curves.
This does not follow immediately from Simpson's construction.
For instance, though $M^\circ_g\subset \overline{M}^{ps}_g$, there is no universal curve over $\overline{M}^{ps}_g$.
The answer for $7/10<\alpha\leq 9/11$ is a corollary to our main theorem, proved in \S\ref{section:construction}.
\begin{theorem}
	For all $\alpha\in\Q\cap(2/3,1]$, $e,r\geq 1$ and $g\geq 3$, there exists a projective variety $\overline{U}_{e,r,g}(\alpha)$ with a canonical projection $\pi:\overline{U}_{e,r,g}(\alpha)\rightarrow\overline{M}_g(\alpha)$ such that the diagram
	\[\xymatrix{
		U_{e,r}(C^\circ_g/M^\circ_g)\ar@{^(->}[r]\ar[d]&\overline{U}_{e,r,g}(\alpha)\ar[d]^{\pi}\\
		M^{\circ}_g\ar@{^(->}[r]&\overline{M}_g(\alpha)
	}\]
	is cartesian and the top row is a compactification of $U_{e,r}(C^\circ_g/M^\circ_g)$, the moduli space obtained by applying Simpson's construction to the universal curve $C^\circ_g\rightarrow M^\circ_g$.
	Over the GIT stable points of $\overline{M}_g(\alpha)$, the points of $\overline{U}_{e,r,g}(\alpha)$ correspond to aut-equivalence classes of slope semi-stable torsion-free sheaves of rank $r$ and degree $e$.
	Moreover, the fiber of $\pi$ over any GIT stable $[C]\in\overline{M}_{g}(\alpha)$ is isomorphic to $U_{e,r}(C)/\aut(C)$. 
	
	\label{thm:universal-moduli-space}
\end{theorem}
The distinction between Theorem~\ref{thm:universal-moduli-space} and Corollary~\ref{cor:universal-moduli-space} is worth emphasizing.
For $\alpha>2/3-\epsilon$, $\overline{\mc{M}}_g(\alpha)$ admits a good moduli space (\cite{afsflip-existence,afsflip-loc,afsflip-proj}), but $\overline{M}_g(\alpha)$ is only a coarse moduli space for $\alpha>7/10$.  In this setting, the statement of the theorem can be refined because the GIT stable points of $\overline{M}_g(\alpha)$ comprise the entire space.
\begin{corollary}
	\label{cor:universal-moduli-space}
	For all $e,r\geq 1$, $g\geq 3$, and $\alpha>7/10$, the points of $\overline{U}_{e,r,g}(\alpha)$ correspond to aut-equivalence classes of slope semi-stable torsion-free sheaves of rank $r$ and degree $e$.
	Moreover, the fiber of $\pi$ over $[C]\in\overline{M}_{g}(\alpha)$ is isomorphic to $U_{e,r}(C)/\aut(C)$. 
	Finally, $\overline{U}_{e,r,g}(\alpha)$ co-represents $\overline{\mc{U}}_{e,r,g}(\alpha)$ (defined in Def.~\ref{def:moduli-problem}).
\end{corollary}

The primary difference in approach between \cite{caporaso1994compactification} and \cite{pandharipande1996} is whether the degeneration of a bundle on a smooth curve is a balanced vector bundle or a torsion-free sheaf.
This distinction is present in the current literature as well: using an approach more in line with \cite{caporaso1994compactification}, analogous compactifications over $\overline{M}_g^{ps}$ have been constructed for $r=1$ in \cite{vivianietal} and then in general in \cite{fringuelli2016picard}.

Additionally, an alternative compactificication to Pandharipande's space $\overline{U}_{e,r,g}$ has been found in \cite{schmitt2004hilbert} by allowing torsion-free sheaves on singular curves which are the pullback of a vector bundle on the stable reduction of the curve.
\begin{rem}
	Over any moduli stack of curves, $\mc{M}$, there is always an Artin stack $\mc{U}$ of slope semi-stable torsion-free sheaves (this follows for instance from \cite[Thm.~2.1]{lieblich2006remarks}).
The above theorem can be framed as demonstrating that for certain $\mc{M}$, $\mc{U}$ admits a good moduli space $U$ with an ample line bundle.

Alternatively, one can also form an intermediate stack by applying Simpson's construction: to any family of curves $\mc{C}\rightarrow S$ over a scheme, $\mc{U}^{Simp}(\mc{C}/S)$ is defined to be $U(\mc{C}/S)$.

Because Simpson's construction is canonical, we may glue over a cover of $\mc{M}$ and form the stack $\mc{U}^{Simp}$, relatively projective over $\mc{M}$.
In particular, if $\mc{M}$ is Deligne--Mumford, so is $\mc{U}^{Simp}$ (by virtue of being relatively projective over Deligne--Mumford), and therefore by the Keel--Mori Theorem it admits a good moduli space which will also be a good moduli space for $\mc{U}$ (see e.g.\ \cite[Tag 050L]{stacks-project}).

It is not immediately obvious that the good moduli space is projective, but in a discussion with the author, Alexeev observed that again because Simpson's construction is canonical and the good moduli space is relatively projective locally, the relative polarizations glue and the good moduli space is projective.

The assumption that $\mc{M}$ is Deligne--Mumford is critical for this construction, and the GIT construction presented here applies in greater generality.

We will consider weakening the Deligne--Mumford condition in the base in future work, as well as whether $\overline{U}_{e,r,g}^{ps}$ is the first step in a ``relative'' log-minimal model program for $\overline{U}_{e,r,g}$.
\end{rem}

We now outline the paper and our strategy for proving the result.
Our strategy is a modification of the strategy employed in \cite{pandharipande1996}, but differs in two significant ways.
First, we do not restrict our construction to Deligne--Mumford stable curves, but consider general Gorenstein curves--this primarily complicates various bounds in the fiberwise GIT problem.
Second, we consider the problem in the context of the Hassett--Keel program, and prove that the same construction essentially applies over any moduli space of curves with a GIT presentation, though less can be said in the presence of GIT semi-stable curves.

For concreteness, our outline is for the case of pseudo-stable curves.
Because $\overline{M}_g(\alpha)$ is a GIT quotient of a Hilbert or Chow scheme for $\alpha > 2/3$ (\cite{hassett2013log}), the same argument goes through more or less immediately.

First, we observe that for a given degree $e$, twisting the sheaves under consideration by an ample line bundle forms an isomorphism with the same moduli problem for some higher degree; in other words, it suffices to assume $e$ is large (see Remark~\ref{rem:moduli-degree-isomorphism} for details).
For $g\geq 3$, let $H_g$ denote the appropriate locus in the Hilbert scheme corresponding to 4-canonically embedded pseudo-stable curves, with universal curve $U_H\subset H_g\times \mathbb{P}^N$.
Let $\nu:U_H\rightarrow\mathbb{P}^N$ denote the projection map.
A given rank $r$ and degree $e$ uniquely determine a Hilbert polynomial, $\Phi(t)$.
To streamline notation, let $n=\Phi(0)$.
We define $\pi:Q^r\rightarrow H_g$ to be the locus in a relative Quot scheme parametrizing sheaves of uniform rank $r$.
A point $\xi\in Q^r$ corresponds to an equivalence class of the following data: a point of $H_g$ corresponding to a curve $C$; a presentation of sheaves $\mathbb{C}^n\otimes \mathcal{O}_{C}\rightarrow E\rightarrow 0$ such that the Hilbert polynomial of $E$ with respect to $\omega_C^{\otimes 4}$ is $\Phi$ and $E$ is of uniform rank.
		  The groups $SL_{N+1}$ and $SL_n$ act naturally on $Q^r$ by changing the coordinates of the curve's embedding and the presentation of the sheaf, respectively.
		  For any $k$, we have an $SL_{N+1}\times SL_n$-linearized ample line bundle $\mathcal{L}_k:=\pi^\ast\mathcal{O}_{H_g}(k)\otimes\mathcal{O}_{\mathfrak{Quot}}(1)$ on $Q^r$.
		  We may therefore define
		  \[\overline{U}_{e,r,g,k}^{ps}:=Q^r/\!\!/_{\mathcal{L}_k}SL_{N+1}\times SL_n.\]

		  A standard variation of GIT argument (see Propositions \ref{prop:pand-loci-containment} and \ref{prop:stable-loci-equality} in Section \ref{sec:vgit}) tells us that when $k\gg 0$, the GIT (semi-)stability of a point of $Q^r$ is entirely determined by the (semi-)stability with respect to the action of $SL_n$.
		  In fact, this is the same as the GIT (semi-)stability of the point with respect to the action of $SL_n$ on the fiber $\mathfrak{Quot}^{\Phi,\omega^{\otimes 4}_{C}}_{\mathbb{C}^n\otimes\mathcal{O}_C,C}$, linearized by the restriction of $\mathcal{O}_{\mathfrak{Quot}}(1)$.
		  It is well-known (e.g., \cite{simpson1994moduli}) that when the linearization on the fiber has sufficiently high degree, GIT (semi-)stability is equivalent to slope (semi-)stability.
		  However, we require a bound on the linearization which holds independent of the curve under consideration, or in other words, for every fiber of $U_H\rightarrow H_g$ simultaneously.
		  		
		  Such a bound is given by \cite[Thm.~4.7]{simpson1994moduli} by applying Simpson's construction to the universal family $U_H\rightarrow H_g$.

Moving forward, we would like to construct analogous moduli spaces over each of the Hassett--Keel moduli spaces and complete the point classification for $\alpha \leq 7/10$.
At the time of writing, the latest results due to \cite{afsflip-existence,afsflip-loc,afsflip-proj} include a classification up to $\alpha=2/3-\epsilon$.
One of the first obstructions to applying the same techniques to classify the points is the existence of strictly semi-stable curves in $\overline{M}_g(\alpha)$ for $\alpha\leq 7/10$.
Additionally, as $\alpha$ varies, the various $\overline{U}_{e,r,g}(\alpha)$ are clearly birational, because they all compactify the moduli of vector bundles over automorphism-free curves; further details about these birational mappings have not yet been worked out.
Finally, the above result demonstrates that $\overline{U}_{e,r,g}^{ps}$ is a ``good moduli space,'' in the sense of \cite{alper2008good}, for the Artin stack, $[Q^{SS}/\!\!/G]$.
This approach to the problem is considered in greater detail in forthcoming work of the author.

Now we briefly outline the structure of the paper.
We establish notation and recall various standard results in section \S\ref{section:preliminaries}.
In section \S\ref{section:construction}, we precisely state the moduli problem of interest, we perform the construction of the compactified universal moduli space in detail, and demonstrate that the constructed space co-represents the moduli functor.
\subsection*{Acknowledgements}
The author would like to take this opportunity to gratefully acknowledge his advisor, Sebastian Casalaina-Martin, for introducing the author to the problem, his guidance, and his invaluable comments in editing and revising this paper.
The author is also indebted to the referees whose thoughtful comments have greatly improved quality of the paper.
\section{Preliminaries}
\label{section:preliminaries}
\subsection{Notations and Conventions}  
Here we fix notation and recall standard useful results which we will use later.
\begin{notation}[Curve]
	A \emph{curve} is a proper connected one-dimensional scheme over the complex numbers.
	The \emph{genus} of a curve $C$ will refer to the arithmetic genus of $C$, $h^1(C,\mathcal{O}_C)$.
\end{notation}
\begin{definition}
	For a reduced curve $C$, the \emph{class of singularity types} of $C$ is defined to be
	\[T=T(C):=\{[\widehat{\O}_{C,x}]:x\in C\},\]
	where $[\widehat{\O}_{C,x}]$ denotes the isomorphism class of the completed local ring $\widehat{\O}_{C,x}$.
	We say two curves $C$ and $C'$ \emph{have the same class of singularity types} if $T(C)=T(C')$.
	Given a set $\mathscr{T}$ of isomorphism classes of complete local rings, we say that a curve \emph{has at worst singularities of type $\mathscr{T}$} if $T(C)\subset\mathscr{T}$.
\end{definition}
\begin{rem}
	As there are only a finite number of singular points in a given reduced curve, and the complete local ring on a smooth point is the completion of a polynomial ring, there are only a finite number of isomorphism classes of rings in $T(C)$.
	This definition is related to the common definition of singularity type of a curve, but does not keep track of the count of each singularity type.
	For example, all singular nodal curves have the same class of singularity types.  
	\end{rem}
Recall the following version of asymptotic Riemann--Roch.
\begin{lemma}[{\cite[Corollary 8, p.~152]{seshadri1982}}]
	\label{lemma:rank-definitions}
	Let $(C,L)$ be a polarized curve with $\deg L=d$.
	For the irreducible components of $C$, $\{C_i\}$, denote by $L_i$ the restriction $L|_{C_i}$.
	Let $d_i=\deg L_i$.
	Then for any coherent sheaf $F$, we have
	\[\chi(F\otimes L^t)=\chi(F)+t\sum_i r_id_i,\]
	where $r_i:=\dim_{k(\eta_i)}F|_{C_i}\otimes k(\eta_i)$ and $\eta_i$ is the generic point of $C_i$.
\end{lemma}
Motivated by this lemma, one makes the following definition of rank and degree of a sheaf on a curve.
\begin{definition}[Rank and Degree]
	Let $(C,L)$ be a polarized curve of genus $g$ with $\deg L=d$ and let $F$ be a coherent sheaf on $C$.
	If $\Phi(t)=\chi(F\otimes L^t)$, the \emph{rank} and \emph{degree} of $F$ \emph{with respect to $L$} are defined so that
	\[\Phi(t)=\deg_LF+\rank_L F\chi(\O_C)+t\rank_L F\deg L\]
	holds.
\end{definition}
	It follows that if $C$ is irreducible the generic rank agrees with $\rank_L$.
	In this case, neither $\rank_L$ nor $\deg_L$ depend on $L$.
\begin{definition}
	A sheaf $F$ on $C$ is said to be of uniform rank if there exists a number $r$ such that for every component $C_i$ of $C$, $\rank F|_{C_i}=r$.
\end{definition}
\begin{rem}
	If $F$ is of uniform rank, then $\rank_L F$ and $\deg_L F$ are both integers and are independent of $L$.
	Indeed, this follows from Lemma \ref{lemma:rank-definitions} because $\sum r_id_i=rd$.
\end{rem}
\begin{rem}
	\label{remark:bound-on-rank}
	In particular, even when $F$ is not of uniform rank, then because $d_i>0$ for every $i$, $r_i\leq rd$.
	We will make use of this fact later in the paper.
\end{rem}

We will make extensive use of the fact that, for a coherent sheaf $F$ of uniform rank and a line bundle $M$, we have
\[\rank_L(F\otimes M)=\rank_L F,\quad \deg_L(F\otimes M)=\deg_L(F)+\rank_L(F)\deg(M).\]

A coherent sheaf $F$ on $C$ is said to be \emph{pure} if for every non-zero subsheaf $F'\subset F$, the dimension of the support of $F'$ is equal to the dimension of the support of $F$.
A coherent sheaf $F$ on $C$ is said to be \emph{torsion-free} if it is pure and the support of $F$ is equal to $C$.
\begin{definition}
	\label{def:slope-stable}
Let $(C,L)$ be a polarized curve and $F$ a torsion-free sheaf on $C$.
$F$ is said to be slope stable (slope semi-stable) with respect to $L$ if for every nonzero, proper subsheaf $0\rightarrow E\rightarrow F$,
\[\frac{\chi(E)}{\sum s_id_i}<(\leq)\frac{\chi(F)}{\sum r_id_i},\]
where $s_i$ and $r_i$ denote the ranks of $E$ and $F$ on each irreducible component of $C$, and $d_i$ is the degree of $L$ restricted to each irreducible component.
\end{definition}
\begin{rem}
If $(C,L)$ is a polarized irreducible curve and $F$ is a vector bundle on $C$, then $F$ is slope-stable (slope-semistable) with respect to $L$ if and only if for each nonzero subsheaf $0\rightarrow E\rightarrow F$,
\[\frac{\deg(E)}{\rank(E)}<(\leq)\frac{\deg(F)}{\rank(F)}.\]
We caution the reader that if $C$ is reducible, then even if we restrict to sheaves of uniform rank so that $\rank_L$ and $\deg_L$ do not depend on $L$, the slope stability condition of Definition \ref{def:slope-stable} does in general depend on $L$, because $E$ is not required to be of uniform rank.
\end{rem}
%
We are interested in the interaction of singularities on a curve with sheaves on the curve.
The following lemma allows us to bound the dimension of certain quotients in terms of analytic invariants of the curve in question.
\begin{lemma}[{\cite[Lemma~7, p.~150]{seshadri1982}}]
	\label{lemma:module-dimension-bound}
	Let $x$ be a point of an irreducible curve $Y$.
	Let $M$ be an $\O_{Y,x}$-module, torsion-free of rank $r$.
	Then
	\[\dim_{\mathbb{C}}(M/\mathfrak{m}_xM)\leq (1+\dim_{\mathbb{C}}\overline{\O}_{Y,x}/\O_{Y,x})\cdot r.\]
\end{lemma}
The statement and proof of Lemma \ref{lemma:module-dimension-bound} assume $Y$ is irreducible.
We have the following bound for reducible $Y$.
\begin{corollary}
	\label{cor:module-dimension-bound-reducible}
	Suppose $(Y,L)$ is a polarized projective curve with $\deg L=d$.
	Let $\mathcal{F}$ be a coherent $\O_Y$-module of multirank $(r_1,\dots,r_p)$.
	Then for any $y\in Y$,
	\begin{equation}
		\label{eqn:module-dimension-bound-reducible}
		\dim_{\mathbb{C}}\left( \mathcal{F}_y/\mathfrak{m}_y\mathcal{F}_y \right)\leq d\delta\max_i (r_i),	\end{equation}
	where
	\[\delta=\max_{A\in T(Y)} (1+\dim_{\mathbb{C}}\overline{A}/A).\]
\end{corollary}
\begin{proof}
	This follows by bounding the restriction of $\mathcal{F}$ to each component of $Y$ (\cite[p.~152]{seshadri1982}) and observing that $d$ is greater than the number of irreducible components of $Y$.
	Indeed, from \cite[p.~152]{seshadri1982}, we have the claim
	\[\dim_{\C} \left(\mathcal{F}_y/\mathfrak{m}_y\mathcal{F}_y\right)\leq \sum r_{l}\delta_{l},\]
	where the summation is over the components of the normalization containing $y$.
	By $r_{l}$, we denote the rank of the stalk of $\mathcal{F}|_{Y_{l},y}$, and $\delta_l$ is the bound of Lemma \ref{lemma:module-dimension-bound}.
	Thus we arrive at the claimed bound by taking the maximum over $l$ of $r_{l}$ and $\delta_l$ and observing that the number of irreducible components must be at most $d$.
\end{proof}
\subsection{Variation of GIT for quotients of products}
\label{sec:vgit}
In this section, we study the properties of quotients of products.
Let $G$ and $H$ be reductive groups.
Let $(X,L_X)$ and $(Y,L_Y)$ be polarized schemes with linearized actions of $G$ on both $X$ and $Y$, and of $H$ on $Y$.
Assume that the actions of $G$ and $H$ commute on $Y$.
Then we have an induced action of $G\times H$ on the product $X\times Y$ given by 
\[(g,h)\cdot(x,y)=(g\cdot x,g\cdot(h\cdot y)).\]
Moreover, we have many linearizations on $X\times Y$ corresponding to 
\[L_{X}^{\otimes a}\boxtimes L_{Y}^{\otimes b}:=\pi_X^\ast L_X^{\otimes a}\otimes\pi_Y^\ast L^{\otimes b}_Y\] 
for all $(a,b)\in\Z^2_{> 0}$.

With several group actions under consideration, we fix some notation: the superscripts $S$ and $SS$ will indicate stability and semi-stability with respect to the product action of $G\times H$.
Stability and semi-stability for $G$ alone will be indicated by superscripts $S_G$ and $SS_G$, and similarly for $H$ alone.
We will also refer to the (semi-)stable locus in $X$ with respect to $G$ as $X^{(S)S_G}$.

Our plan, following \cite{pandharipande1996}, is to shift the weight of the polarization almost entirely to $X$.
This, we will show, reduces the stability condition for $G\times H$ on $X\times Y$ to the stability condition for $H$ on $Y$.
The following key propositions, understood in the context of variation of GIT, make this precise:
\begin{prop}
	\label{prop:pand-loci-containment}
	Let $\pi_X:X\times Y\rightarrow X$ be the natural projection map.
	Then for $a/b\gg 0$, we have, with respect to the linearization $L_X^{\otimes a}\boxtimes L_Y^{\otimes b}$ on $X\times Y$,
	\[\pi^{-1}_X(X^{S_G})\subset (X\times Y)^{S_{G}}_{[a,b]}\subset (X\times Y)^{SS_{G}}_{[a,b]}\subset\pi^{-1}_X(X^{SS_G}).\]
\end{prop}
\begin{proof}
	This is a standard result in variation of GIT; see e.g. \cite[Lemma 4.1]{thaddeus1996geometric}.
	This particular formulation is equivalent to Propositions 7.1.1 and 7.1.2 in \cite{pandharipande1996}.
\end{proof}
\begin{prop}[{\cite[Prop.~8.2.1]{pandharipande1996}}]
	\label{prop:stable-loci-equality}
	Let $Q\subset \pi^{-1}_X(X^{S_G})$ be a closed subscheme.
	Then for $a,b$ as in Proposition \ref{prop:pand-loci-containment}, we have
	\[Q^{S_H}_{[a,b]}=Q^{S}_{[a,b]} \text{ and } Q^{SS_H}_{[a,b]}=Q^{SS}_{[a,b]}.\]
\end{prop}
\begin{proof}
\ifdetailed
\else
We sketch a variation of GIT argument here.
An explicit proof in coordinates can also be found in \cite[Prop.~8.2.1]{pandharipande1996}.
We only prove the statement for stable loci; the semi-stable case is identical.
\fi

	To begin, certainly, $Q^S_{[a,b]}\subset Q^{S_H}_{[a,b]}$, so it suffices to demonstrate the opposite inclusion.
	For this, let $\lambda$ be a one-parameter subgroup of $G\times H$, with components $\lambda_G$ and $\lambda_H$.

	Recall briefly that the Hilbert--Mumford index for $(x,y)\in Q_{[a,b]}$ with respect to $\lambda$ is obtained by taking the limit of $\lambda\cdot(x,y):=(x,y)_0$ as the parameter approaches 0.
	The limit $(x,y)_0$ is fixed by $\lambda$ and therefore $\lambda$ acts on the fiber of $L^{\otimes a}_X\boxtimes L^{\otimes b}_Y$ over $(x,y)_0$.
	The Hilbert--Mumford index is the character of this action.

	Let $\mu$ denote the Hilbert--Mumford index and fix $(x,y)\in Q^{S_H}_{[a,b]}$.
	From local arguments,
we have
	\begin{equation}
		\mu^{L_X^a\boxtimes L_Y^b}( (x,y),\lambda)=a\mu^{L_X}(x,\lambda_G)+b\mu^{L_Y}(y,\lambda_G)+b\mu^{L_Y}(y,\lambda_H).
		\label{eqn:hm-index}
	\end{equation}
	From the first inclusion of Proposition \ref{prop:pand-loci-containment}, and the assumption that $Q\subset \pi^{-1}_X(X^{S_G})$, it follows that for $\lambda_G\neq 1$ the sum of the first two terms on the right-hand side of \eqref{eqn:hm-index} is negative.
	If $\lambda_G=1$, the first two terms sum to 0.
	From the assumption that $(x,y)\in Q^{S_H}_{[a,b]}$, it follows that the last term is also negative.
	Therefore, $(x,y)\in Q^S_{[a,b]}$.
\end{proof}

%
\section{Construction of $\overline{U}_{e,r,g}(\alpha)$}
\label{section:construction}
In this chapter, we construct the compactified universal moduli space as a GIT quotient.
First, in \S~\ref{section:moduli-problem} we state the moduli problem.
The GIT construction takes place in \S~\ref{section:main-construction}, and then we prove that the GIT quotient co-represents the moduli functor over the locus of stable curves in \S~\ref{section:corepresent}.
\subsection{The moduli problem}
\label{section:moduli-problem}
Here, we describe the moduli functor of sheaves we wish to study.
Throughout, we assume that $g\geq 2$.
The notion of $\alpha$-stability is developed in \cite{afsflip-loc} to describe the various stability conditions that arise in the Hassett--Keel program.
Defined for $2/3-\epsilon<\alpha\leq 1$, $\alpha$-stable curves are, in particular, reduced curves with at worst type $A_4$ singularities.
We refer the reader to \cite[Def.~2.5]{afsflip-loc} for the complete definition of $\alpha$-stability, but reproduce the main points here for convenience.
\begin{definition}
	\label{def:alpha-stability}
	A curve has a singularity of type $A_n$ at $p\in C$ if, \'etale locally near $p$ $C$ can be described as 
	\[y^2=x^{n+1}.\]
	For $\alpha\in(2/3-\epsilon,1]$, we say that a curve $C$ is \emph{$\alpha$-stable} if $\omega_C$ is ample and:
	\begin{itemize}
		\item[] For $\alpha\in(9/11,1)$: $C$ has only $A_1$-singularities.
		\item[] For $\alpha=9/11$: $C$ has only $A_1$,$A_2$-singularities.
		\item[] For $\alpha\in(7/10,9/11)$: $C$ has only $A_1$,$A_2$-singularities, and does not contain:
			\begin{itemize}
				\item $A_1$-attached elliptic tails.
			\end{itemize}
		\item[] For $\alpha=7/10$: $C$ has only $A_1$,$A_2$,$A_3$-singularities, and does not contain:
			\begin{itemize}
				\item $A_1$,$A_3$-attached elliptic tails.
			\end{itemize}
		\item[] For $\alpha\in(2/3,7/10)$: $C$ has only $A_1$,$A_2$,$A_3$-singularities, and does not contain:
			\begin{itemize}
				\item $A_1$,$A_3$-attached elliptic tails.
				\item $A_1/A_1$-attached elliptic tails.
			\end{itemize}
	\end{itemize}
\end{definition}
For $\alpha\in(7/10,9/11]$, the reader may also refer to Definition~\ref{def:pseudo-stable}.
\begin{definition} 
	\label{def:moduli-problem}
	Let $e$, $r$, and $g$ be integers such that $r\geq 1$ and $g\geq 2$.
	Let $\alpha\in (2/3,1]\cap\Q$.
	The functor $\overline{\mathcal{U}}_{e,r,g}(\alpha)$ associates to each scheme $S$ the following set of equivalence classes of data:
	\begin{itemize}
		\item A family of genus $g$ $\alpha$-stable curves $\mu:\mathcal{C}\rightarrow S$; i.e., a flat proper morphism such that every geometric fiber is an $\alpha$-stable curve of genus $g$.
		\item A coherent sheaf $\mathcal{F}$ on $\mathcal{C}$, flat over $S$, such that on geometric fibers $\mathcal{F}$ is slope-semistable, torsion-free of uniform rank $r$ and degree $e$.
	\end{itemize}
	Two pairs $(\mu:\mathcal{C}\rightarrow S,\mathcal{F})$ and $(\mu':\mathcal{C}'\rightarrow S,\mathcal{F}')$, are equivalent if there exists an $S$-isomorphism $\phi:\mathcal{C}\rightarrow\mathcal{C}'$ and a line bundle $L$ on $S$ such that $\mathcal{F}\cong\phi^\ast\mathcal{F}'\otimes\mu^\ast L$.
\end{definition}
\begin{rem}
	Recall that for $\alpha > 2/3-\epsilon$, the moduli space of $\alpha$-stable curves, $\bar{\mathscr{M}}_g(\alpha)$, admits a good moduli space, $\overline{M}_g(\alpha)$ (\cite{afsflip-existence}).
	For $\alpha > 7/10$, $\overline{M}_g(\alpha)$ is in fact a coarse moduli space.
	For $7/10<\alpha\leq 9/11$, $\overline{M}_g(\alpha)$ is isomorphic to $\overline{M}_g^{ps}$, Schubert's moduli space of pseudo-stable curves.
There is a divisorial contraction of coarse moduli spaces $\overline{M}_g\rightarrow\overline{M}_g^{ps}$ sending Deligne--Mumford stable curves with an elliptic tail to cuspidal curves (\cite{hassett-hyeon2009}).
\end{rem}
We will also require the so-called ``fiberwise'' moduli functor.
\begin{definition} 
	Let $(C,L)$ be any polarized curve.
	The functor $\mathcal{U}_{e,r}(C)$ associates to each scheme $S$ the set of equivalence classes (in the sense of Def.~\ref{def:moduli-problem}) of sheaves $\mathcal{F}$ on $S\times C$, flat over $S$, such that for each $s\in S$, $\mathcal{F}_s$ is slope semi-stable and torsion-free of uniform rank $r$ and degree $e$.
\end{definition}

\begin{rem}
	\label{rem:moduli-degree-isomorphism}
	There is an isomorphism of functors
	\begin{equation}
		\label{eqn:moduli-iso}
		\overline{\mathcal{U}}_{e,r,g}(\alpha)\rightarrow\overline{\mathcal{U}}_{e\pm(2g-2),r,g}(\alpha),
	\end{equation}
	\[(\mu:\mathcal{C}\rightarrow S,\mathcal{F})\mapsto (\mu:\mathcal{C}\rightarrow S,\mathcal{F}\otimes\omega_{\mathcal{C}/S}^{\otimes\pm 1}),\]
	and similarly we have
	\[
		\mathcal{U}_{e,r}(C)\xrightarrow{\sim}\mathcal{U}_{e\pm \deg L,r}(C).
		\]
	As a result, it suffices to study the moduli functors for large $e$.
\end{rem}

\subsection{GIT construction of the compactified universal moduli space}
\label{section:main-construction}
We now construct a GIT quotient which we will see in \S~\ref{section:corepresent} co-represents the restriction of the moduli functor $\overline{\mathcal{U}}_{e,r,g}(\alpha)$ to GIT-stable curves in $\overline{\mathcal{M}}_g(\alpha)$.
\begin{proof}[Proof of Theorem~\ref{thm:universal-moduli-space}]
We have broken the proof into several parts.
The construction and point classification are carried out in Proposition~\ref{prop:main-git-properties}.
The classification of the fibers over stable curves is carried out in Proposition~\ref{prop:fibers}.
The inclusions in the diagram are a consequence of the description of the fibers.
\end{proof}

\begin{prop}
	\label{prop:main-git-properties}
	Using the notation of Theorem~\ref{thm:universal-moduli-space}, there exists a projective variety $\overline{U}_{e,r,g}(\alpha)$ with a canonical projection $\pi:\overline{U}_{e,r,g}(\alpha)\rightarrow \overline{M}_g(\alpha)$.
	The points of $\overline{U}_{e,r,g}(\alpha)$ lying over $GIT$-stable curves correspond to aut-equivalence classes of slope semi-stable torsion-free sheaves of rank $r$ and degree $e$.
\end{prop}
\begin{proof}
We proceed with the proof in three parts: first we set up the GIT problem, then we proceed with the construction of the moduli space, and last we classify the orbit closures over GIT-stable curves.
For the sake of concision, the statements and proofs of various independent supporting arguments will be found after this proof.
\vskip .1in
\noindent\textbf{Part 1 - Setup}
From \cite{hassett2013log} and \cite{afsflip-existence}, for $\alpha > 2/3$ there is a scheme $H_g$, either a Hilbert scheme or Chow scheme depending on $\alpha$, equipped with the action of a reductive group $G$, along with a $G$-linearized line bundle $\O_{H_g}(1)$ such that
\[\overline{M}_g(\alpha)\cong H_g/\!\!/_{\O_{H_g}(1)}G.\]

Let $U_H\hookrightarrow H_g\times\P^N$ be the universal curve over $H_g$.
Let $\nu:U_H\rightarrow\P^N$ be the projection map.
Define $d=\deg \nu^\ast\O_{\P^N}(1)|_C$ for any curve $C\in H_g$.

We will construct our moduli space using a relative Quot scheme.
Specifically, let
\[Q\subset\mathfrak{Quot}_{\C^n\otimes\mathcal{O}_{U_H},U_H,H_g}^{\nu^{\ast}\mathcal{O}_{\P^N}(1),\Phi},\]
be the locus of quotients of uniform rank, where $\Phi$ is a Hilbert polynomial with respect to $\nu^{\ast}\mathcal{O}_{\P^N}(1)$ ensuring that all parametrized sheaves have rank $r$ and degree $e$.
A point $\xi\in Q$ corresponds to an equivalence class of the following data:
		  \begin{itemize}
			  \item a point of $H_g$ corresponding to an $\alpha$-stable curve $C$ embedded in projective space by $\nu^\ast\O_{\P^N}(1)|_C$;
			  \item a presentation of sheaves $\mathbb{C}^n\otimes \mathcal{O}_{C}\rightarrow E\rightarrow 0$ such that the Hilbert polynomial of $E$ with respect to $\nu^\ast\O_{\P^N}(1)|_C$ is $\Phi$ (i.e., $\deg E=e$ and $\rank E=r$) and $E$ is of uniform rank.
		  \end{itemize}
It is well-known that $Q$ is a union of connected components, but for lack of a reference for our specific case, we establish this in Lemma~\ref{lemma:uniform-rank-clopen} below.

The action of $G$ on $H_g$ lifts naturally to an action on $Q$.
Moreover, $Q$ is equipped with an action of $SL_n$ by changing coordinates in $\C^n\otimes \mathcal{O}_{U_H}$.
These two actions commute and therefore induce an action of $G\times SL_n$ on $Q$.
In order to apply the results from variation of GIT above (specifically, Prop.~\ref{prop:stable-loci-equality}), we express the problem in terms of a quotient of a product: there is a closed immersion respecting the group actions
\[Q\subset H_g\times\mathfrak{Quot}_{\C^n\otimes\mathcal{O}_{U_H},U_H,H_g}^{\nu^{\ast}\mathcal{O}_{\P^N}(1),\Phi}.\]
We now recall the very ample line bundle $\O_{\mathfrak{Quot}}(1)$ on the Quot scheme. 
Recall from the construction of the Quot scheme that tensoring a quotient by powers of an ample line bundle, e.g.\ the relative dualizing sheaf,
\[(\O_C^n\rightarrow E\rightarrow 0)\mapsto (\O_C^n\otimes\omega^{t}\rightarrow E\otimes\omega^{t}\rightarrow 0)\]
and applying global sections defines a rational map into a Grassmannian.
For $t\gg 0$, this becomes an embedding.
By applying Simpson's construction to $U_H\rightarrow H_g$, there is a number $t(d,g,r,e)$, depending only on $d$, $g$, $r$, and $e$ which will select our polarization on the Quot scheme.
Fix $b=t$ in Prop.~\ref{prop:stable-loci-equality} and let $a=k$ be the least integer such that the conclusion of the proposition holds.
Define $\O_{\mathfrak{Quot}}(t)$ to be the pullback of the very ample line bundle on the Grassmannian from the Pl\"ucker embedding.
Let
\[\mathscr{L}_{k,t}=\O_{H_g}(k)\boxtimes \O_{\mathfrak{Quot}}(t),\]
where $\boxtimes$ denotes the tensor product of the respective pullbacks.
The line bundle $\mathscr{L}_{k,t}$ admits a $G\times SL_n$-linearization.
\vskip .1in
\noindent\textbf{Part 2 - Construction of $\overline{U}_{e,r,g}(\alpha)$}
We now have everything required to define our GIT quotient:
\[\overline{U}_{e,r,g}(\alpha):=Q/\!\!/_{\mathscr{L}_{k,t}}(G\times SL_n).\]
The reader will recall that a description for large $e$ suffices because of the isomorphism between moduli functors for sheaves of different degrees described in~\eqref{eqn:moduli-iso}.

\vskip .1 in 


First, we observe that for $e\gg 0$, all slope semi-stable sheaves (with respect to the canonical polarization) appear in $Q$.
This is an essentially well-known boundedness statement, whose proof we omit (see for instance \cite[Cor.~1.7.7]{huybrechtslehn2010} and observe that the regularity of slope-semistable sheaves is uniformly bounded in families).

Let $Q_G$ denote the pre-image of $H_g^{S}$ under the projection morphism.
Because we selected $k/t$ so that the conclusion of the variation of GIT result from Proposition~\ref{prop:stable-loci-equality} holds, the $G\times SL_n$-stable (semi-stable) locus of $Q_G$ with respect to the linearization $\mathscr{L}_{k,t}$ is equal to the $SL_n$-stable (semi-stable) locus.
In other words, this implies that if $C$ is GIT-stable, then a pair $(C,F)$ is GIT (semi-)stable if and only if $F$ is GIT (semi-)stable as a point of the fiber of $Q_G$ over $[C]\in H_g$ with respect to the linearization induced by the restriction of $\mathscr{L}_{k,t}$.
Note that the fiber of $Q_G$ over $[C]\in H_g$ is naturally embedded in the Quot scheme $\mathfrak{Quot}^{\nu^\ast\O_{\P^N}(1)|_C,\Phi}_{\C^n\otimes\O_C,C}$.

If we can establish that GIT (semi-)stability with respect to the restriction of $\mathscr{L}_{k,t}$ is equivalent to slope (semi-)stability (with respect to the specified polarization) in each fiber simultaneously, then we will have demonstrated that $(C,F)$ is GIT (semi-)stable if and only if $F$ is slope (semi-)stable. 

We have seen however that this is true by applying Simpson's result to the universal family $U_H\rightarrow H_g$: slope stability and GIT stability agree asymptotically.

We have thus established that for $k\gg 0$ and a GIT-stable curve $C$, the GIT (semi-)stability of $(C,F)$ with respect to $G\times SL_n$ is equivalent to slope (semi-)stability of $F$. 

Now we construct the projection map $\pi$.
The morphism 
\[Q^{SS}\rightarrow H^{SS}_g\rightarrow \overline{M}_g(\alpha)\]
is equivariant with respect to the group action, and so by the universal property of the GIT quotient, induces a morphism
\[\overline{U}_{e,r,g}(\alpha)\xrightarrow{\pi} \overline{M}_g(\alpha),\]
sending a curve and a sheaf to the underlying curve.
\vskip .1 in 
\noindent\textbf{Part 3 - Orbit closures.}

First, observe that because $Q$ is a union of connected components (Lemma~\ref{lemma:uniform-rank-clopen}), the space $\overline{U}_{e,r,g}(\alpha)$ is the GIT quotient of a closed subset of a projective scheme, and is therefore projective.
By our construction above, Proposition~\ref{prop:stable-loci-equality} guarantees that the stable and semi-stable loci over GIT-stable curves are completely described by the fiberwise stable and semi-stable loci, described by Simpson.
The locus of slope semi-stable vector bundles on a smooth curve is open because it is the preimage under $\pi$ of an open subset of $\overline{M}_g(\alpha)$.

Next, we classify the orbit closures over the GIT-stable curves.
Let $Q^{SS}_G\subset Q^{SS}$ be the pre-image of $H_g^{S}$ under the projection morphism, let $\xi\in Q^{SS}_G$, and suppose that $\bar{\xi}\in Q^{SS}_G$ lies in the orbit closure of $\xi$.
It is immediate that $\pi(\bar{\xi})$ is in the orbit closure of $\pi(\xi)$.
Thus, if $\xi$ corresponds to $(C,F)$ and $\bar{\xi}$ corresponds to $(\overline{C},\overline{F})$, we see that $C$ and $\overline{C}$ are projectively equivalent.
The $G$-orbit closure of $\bar{\xi}$ consists of the images of $\overline{F}$ under projective automorphisms of $\overline{C}$.
On the other hand, the $SL_n$-orbit closure of $\xi$ is known (e.g., \cite[Thm. 1.21]{simpson1994moduli}) to consist of sheaves $E$ which are aut-equivalent to $F$.

We will demonstrate that these two orbit closures intersect, which will prove that $\xi$ and $\bar{\xi}$ are aut-equivalent.
Consider a path
\[\gamma=(\gamma_1,\gamma_2):\Delta\setminus\{p\}\rightarrow G\times SL_n,\]
such that
\[\lim_{z\rightarrow p}\gamma(z)\cdot\xi=\bar{\xi}.\]
Composing the path with the group action induces
\[\mu:\Delta\setminus\{p\}\rightarrow Q_G,\quad \mu(z):=\gamma_2(z)\cdot\xi.\]
As $Q$ is projective, $\mu$ extends to $\Delta$.
Notice that $\mu(p)$ is in the $SL_{n}$-orbit closure of $\xi$.
If we can demonstrate that $\mu(p)$ is also in the $G$-orbit closure of $\bar{\xi}$, then we are done.
We have
\[\lim_{z\rightarrow p}\gamma_1(z)\cdot\mu(p)=\lim_{z\rightarrow p}\gamma_1(z)\cdot\lim_{z\rightarrow p}(\gamma_2(z)\cdot \xi)=\lim_{z\rightarrow p}(\gamma_1(z)\cdot\gamma_2(z)\cdot\xi)=\bar{\xi}.\]
This completes the proof of the theorem.
\end{proof}
\begin{rem}
	The classification of the orbit closures fails for GIT strictly semi-stable curves because the argument relies on a description of the GIT semi-stable points in the fiber over the curve.
\end{rem}

For lack of a better reference, we include the following lemma to establish that the locus of sheaves of uniform rank in the Quot scheme above is a union of connected components.  It is similar to \cite[Lemma~8.1.1]{pandharipande1996}, with the difference that we work with arbitrary families of curves instead of Deligne--Mumford stable curves.
We note in passing that the result holds in greater generality.
In particular, with an eye towards future work, the result applies to other loci in the Hilbert scheme arising in the Hassett--Keel program (e.g., \cite{hassett2013log}).
\begin{lemma}
	\label{lemma:uniform-rank-clopen}
	Let $g\geq 2$ and $r$ be integers.
	Define $\Phi(t)= e+r(1-g) + drt$ and let $n=\Phi(0)$.
	Let $\kappa:\mathcal{C}\rightarrow\mathcal{B}$ be a projective, flat family of genus $g$ curves parametrized by an irreducible curve.
	Because $\kappa$ is a family of projective curves, $\mathcal{C}$ is equipped with a relatively ample line bundle $\mathcal{L}$.
	Assume the relative degree of $\mathcal{L}$ is at least $d$.
	Define 
	\[Q\subset\mathfrak{Quot}^{\mathcal{L},\Phi}_{\C^n\otimes\O_{\mathcal{C}},\mathcal{C},\mathcal{B}}\] 
	to be the subset corresponding to quotients
	\[\C^n\otimes\O_C\rightarrow E\rightarrow 0,\]
	where $E$ has uniform rank $r$ on $C$.
	Then the subscheme $Q$ is open and closed in $\mathfrak{Quot}^{\mathcal{L},\Phi}_{\C^n\otimes\O_{\mathcal{C}},\mathcal{C},\mathcal{B}}$.
\end{lemma}
\begin{proof}
	Let $\mathcal{E}$ be a $\kappa$-flat coherent sheaf.

	Suppose there exists a $b^\ast\in \mathcal{B}$ such that $\mathcal{E}_{b^\ast}$ has uniform rank $r$ on $\mathcal{C}_{b^\ast}=C$.
	Let $\{\mathcal{C}_i\}$ be the irreducible components of $\mathcal{C}$.
	The morphism $\kappa$ is flat and surjective of relative dimension 1, and so each $\mathcal{C}_i$ contains a component of $C$.
	By the semi-continuity of 
	\[r(z):=\dim_{k(z)}(\mathcal{E}\otimes k(z)),\]
	there is an open set $U_i\subset\mathcal{C}_i$ where $r(z)\leq r$.

	The set $U=\cap_i\kappa(U_i)\subset\mathcal{B}$ is open, and has the property that for every $b\in U$ the rank of $\mathcal{E}_b$ at the generic point of each irreducible component of $\mathcal{C}_b$ is at most $r$.
	We will show that $U=\mathcal{B}$ and conclude that $\mathcal{E}_b$ is of uniform rank for every $b\in U$.

	By way of contradiction, suppose that there exists a $b'\in \mathcal{B}$ such that $\mathcal{E}_{b'}$ is not of uniform rank $r$.
	Then again by semi-continuity, there is an $i$ so that $r(z)<r$ on an open $W\subset\mathcal{C}_i$:
	As $\mathcal{E}$ is flat over $\mathcal{B}$, the Hilbert polynomial of $\mathcal{E}_b$ is constant.
	In particular, the coefficient $rd$ of $t$ is constant.
	By Remark~\ref{remark:bound-on-rank}, $\sum_j r_jd_j=rd$, where $r_j$ is the generic rank on the $j$-th component of $\mathcal{C}_{b'}$.
	If $\mathcal{E}_{b'}$ is not of uniform rank, then some $r_j$ is greater than $r$ and some $r_i$ is less than $r$.
	Fixing the component $\mathcal{C}_i$ containing that component, we may appeal to upper semi-continuity and see that there is an open subset with rank bounded by $r_i<r$.

	But for any $b\in U\cap \kappa(W)$, the multiranks of $\mathcal{E}_b$ is at most $r$ on each component and strictly less than $r$ on at least one component.
	By Lemma~\ref{lemma:rank-definitions}, $\mathcal{E}_b$ cannot have Hilbert polynomial $\Phi(t)$, a contradiction.

	Thus, there was no such $b'$, and so for every $b\in\mathcal{B}$, $\mathcal{E}_b$ has uniform rank $r$, proving the lemma.
\end{proof}

\subsection{The quotient $\overline{U}_{e,r,g}(\alpha)$ co-represents $\overline{\mathcal {U}}_{e,r,g}(\alpha)$ over stable curves}
\label{section:corepresent}
We introduce a piece of notation for the following.
The functor $\overline{\mc{U}}_{S,e,r,g}(\alpha)$ is the restriction of $\overline{\mc{U}}_{e,r,g}(\alpha)$ to GIT-stable curves.
The same notation indicates the restriction of $\overline{U}_{e,r,g}(\alpha)$.
\begin{theorem}
	For any $e,r,g$ with $r\geq 1$ and $g\geq 2$, the scheme $\overline{U}_{S,e,r,g}(\alpha)$ co-represents the functor $\overline{\mathcal{U}}_{S,e,r,g}(\alpha)$.
	\label{thm:moduli}
\end{theorem}
Recall that to say $\overline{U}_{S,e,r,g}(\alpha)$ co-represents the functor $\overline{\mathcal{U}}_{S,e,r,g}(\alpha)$ is to say that $\overline{U}_{S,e,r,g}(\alpha)$ is initial with respect to morphisms from $\overline{\mathcal{U}}_{S,e,r,g}(\alpha)$ to schemes:
\[\xymatrix{
	\overline{\mathcal{U}}_{S,e,r,g}(\alpha)\ar[r]\ar[rd] & \overline{U}_{S,e,r,g}(\alpha)\ar@{-->}[d]\\
	& Z.  }\]
\begin{proof}
	First, note that by our definitions and the isomorphism 
	\[\overline{\mathcal{U}}_{S,e,r,g}(\alpha)\cong\overline{\mathcal{U}}_{S,e\pm(2g-2),r,g}(\alpha),\] 
	it suffices to prove the claim for $e>E(d,g,r,T)$, where $d$ is the degree of the polarization on the family of curves and $T$ is the class of singularities for $\alpha$-stable curves.
Now, we construct a natural transformation 
\[\phi:\overline{\mathcal{U}}_{e,r,g}(\alpha)\rightarrow\Hom(-,\overline{U}_{e,r,g}(\alpha)).\]
Let $e>E(d,g,r,T)$.
For a scheme $S$, let $(\mu,\mathcal{C},\mathcal{F})\in \overline{\mathcal{U}}_{e,r,g}(\alpha)(S)$.
The sheaf $\mu_\ast \nu^\ast\mathcal{O}_{\P^N}(1)$ is locally free of rank $N+1$.
Additionally, as we have taken $e$ sufficiently large, for all $s\in S$ 
$\mathcal{F}_s$ is generated by global sections.
This is a consequence of Corollary~\ref{cor:module-dimension-bound-reducible} and the slope semi-stability of $\mathcal{F}_s$.
In particular, $H^0(\mathcal{C}_s,\mathcal{F}_s)=\chi(\mathcal{F}_s)=:n$.
Thus $\mu_\ast\mathcal{F}$ is locally free of rank $n$.
Let $\{W_i\}$ be an open cover of $S$ trivializing both $\mu_\ast\nu^\ast\mathcal{O}_{\P^N}(1)$ and $\mu_\ast\mathcal{F}$:
\[\alpha_i:\C^{N+1}\otimes\mathcal{O}_{W_i}\xrightarrow{\cong} \mu_\ast\nu^\ast\mathcal{O}_{\P^N}(1)|_{W_i},\]
\[\beta_i:\C^{n}\otimes\mathcal{O}_{W_i}\xrightarrow{\cong} \mu_\ast\mathcal{F}|_{W_i}.\]
If $V_i=\mu^{-1}(W_i)$, then pulling back we obtain compositions
\[\C^{N+1}\otimes\mathcal{O}_{V_i}\xrightarrow{\cong}\mu^\ast(\mu_\ast\nu^\ast\mathcal{O}_{\P^N}(1)|_{V_i})\rightarrow\nu^\ast\mathcal{O}_{\P^N}(1)|_{V_i},\]
\[\C^n\otimes\mathcal{O}_{V_i}\xrightarrow{\cong}\mu^{\ast}(\mu_{\ast}\mathcal{F}|_{V_i})\rightarrow\mathcal{F}|_{V_i}.\]
The second morphisms, and hence the compositions, are surjective because both $\nu^\ast\O_{\P^N}(1)$ and $\mathcal{F}_s$ are globally generated; the former because it is very ample and the latter by our comment above.
Moreover, by construction, the induced maps on global sections are surjective as well.
A dimension count shows that they are isomorphisms.
By the universal property of $Q$, we obtain morphisms $\phi_i:W_i\rightarrow Q$.

We now pause to restate what we have established about the fiberwise behavior of $\mathcal{F}$:
\begin{itemize}
	\item $\mathcal{F}$ is fiberwise slope-semistable and torsion-free of uniform rank
	\item the fiberwise presentation of $\mathcal{F}$ induces an isomorphism on global sections.
\end{itemize}
From \cite{simpson1994moduli}, we know that we may \emph{uniformly} select a lower bound on $t$ depending only on $d,g,r,e$ such that for larger $t$, such families of sheaves are in the semi-stable locus of $Q$.
In other words, $\phi_i(W_i)\subset Q^{SS}$. 

As $\phi_i|_{W_i\cap W_j}$ differs from $\phi_j|_{W_i\cap W_j}$ precisely by the trivializations defined above, we obtain a well-defined morphism 
\[\phi:S\rightarrow \overline{U}_{S,e,r,g}(\alpha).\]
The naturality of the universal property of $Q$ implies that the defined $\phi$ is also natural.

The proof is complete pending the universality of $\phi$.
This is, however, a straightforward diagram chase and is left to the reader.
\end{proof}
Now we study the fibers of $\pi:\overline{U}_{S,e,r,g}(\alpha)\rightarrow\overline{M}_{S,g}(\alpha)$ over GIT-stable curves.
\begin{lemma}
	Fix an GIT-stable curve $C$.
	The \'etale sheafification of the fiber functor $\overline{\mathcal{U}}_{e,r,g}(\alpha)\times_{\overline{\mathcal{M}}_g(\alpha)}[C]$ is isomorphic to the stack quotient $\left[\mathcal{U}_{e,r}(C)/\aut(C)\right]$.
	\label{lem:fiber-1}
\end{lemma}
\begin{proof}
	The functor $\mathcal{U}_{e,r}(C)$ parametrizes families of sheaves on trivial families of curves isomorphic to $C$.	
	By definition, the stack quotient $\left[\mathcal{U}_{e,r}(C)/\aut(C)\right]$ parametrizes families of sheaves on isotrivial families of curves isomorphic to $C$.

	On the other hand, by definition, $\overline{\mathcal{U}}_{e,r,g}(\alpha)\times_{\overline{\mathcal{M}}_g(\alpha)}[C]$ parametrizes sheaves on isotrivial families of curves isomorphic to $C$.

	The two moduli functors therefore define isomorphic algebraic stacks.
\end{proof}
\begin{prop}
	Let $C$ be a GIT-stable curve.
	Then the fiber 
	\[\overline{U}_{e,r,g}(\alpha)\times_{\overline{M}_{g}(\alpha)}[C]\]
	is isomorphic to
	\[U_{e,r}(C)/\aut(C).\]
	\label{prop:fibers}
\end{prop}
\begin{proof}
	First, because $\overline{U}_{e,r,g}(\alpha)$ is a universal categorical quotient (see \cite{mumford1994geometric}), for a GIT-stable curve $C$ the fiber $\overline{U}_{e,r,g}(\alpha)\times_{\overline{M}_g(\alpha)}[C]$ co-represents the fiber $\overline{\mathcal{U}}_{e,r,g}(\alpha)\times_{\overline{\mathcal{M}}_g(\alpha)}[C]$.

	Hence, by Lemma~\ref{lem:fiber-1}, $\overline{U}_{e,r,g}(\alpha)\times_{\overline{M}_g(\alpha)}[C]$ co-represents  $\left[\mathcal{U}_{e,r}(C)/\aut(C)\right]$.
	If we can establish that $U_{e,r}(C)/\aut(C)$ also co-represents $\left[\mathcal{U}_{e,r}(C)/\aut(C)\right]$, then the claim is proved.

	But this immediate because $U_{e,r}(C)$ co-represents $\mathcal{U}_{e,r}(C)$.
\end{proof}

\section{Properties of $\overline{U}_{e,r,g}^{ps}$}
\subsection{The irreducibility of $\overline{U}_{e,r,g}^{ps}$}
Before defining the universal moduli functor, let us recall the definition of a pseudo-stable curve.
\begin{definition}\label{def:pseudo-stable}
A projective curve is pseudo-stable if
\begin{itemize}
	\item it is connected, reduced, and has only nodes and cusps as singularities;
	\item every subcurve of genus one meets the rest of the curve in at least two points;
	\item the canonical sheaf of the curve is ample.
\end{itemize}
Given a scheme $S$, a \emph{family of genus $g$ pseudo-stable curves parametrized by $S$} is a morphism $f:C\rightarrow S$, where $f$ is a flat and proper morphism such that every geometric fiber is a pseudo-stable curve of genus $g$.
Two families $f:C\rightarrow S$ and $g:D\rightarrow S$ are isomorphic if they are isomorphic over $S$.
Recall the moduli functor $\bar{\mathscr{M}}^{ps}_g$ which associates to a scheme $S$ the set of all families of genus $g$ pseudo-stable curves parametrized by $S$ modulo isomorphism.
\end{definition}

\label{section:irreducible}
Now, we establish the irreducibility of $\overline{U}^{ps}_{e,r,g}$.
\ifdetailed
\else
The following lemma extends Lemma 9.2.3 of \cite{pandharipande1996} and lays the groundwork for a deformation-theoretic argument.
\fi
\begin{lemma}
	Let $C$ be a pseudo-stable curve of genus $g$.
	Let $E$ be a slope-semistable torsion-free sheaf of uniform rank $r$ on $C$.
	Then there exists a family $\mu:\mathcal{C}\rightarrow\Delta_0$ and a $\mu$-flat coherent sheaf $\mathcal{E}$ on $\mathcal{C}$ such that:
	\begin{enumerate}
		\item $\Delta_0$ is a pointed curve.
		\item $\mathcal{C}_0\cong C$, and for every $t\neq 0$, $\mathcal{C}_t$ is a complete, nonsingular, irreducible genus $g$ curve.
		\item $\mathcal{E}_0\cong E$, and for every $t\neq 0$, $\mathcal{E}_t$ is a slope-semistable torsion-free sheaf of rank $r$.
	\end{enumerate}
	\label{lemma:irreducible-3}
\end{lemma}
\begin{proof}
	Let $z\in C$ be a singular point.
	Because $E$ is torsion-free of uniform rank $r$, we have
	\[E_z\cong \O_z^{\oplus a_z}\oplus \mathfrak{m}_z^{\oplus r-a_z},\]
	where $a_z$ is an integer determined by $E$ called the local semirank of $E$ (see \cite{martinsetal-torsionfree2012}).
	This follows when $C$ has a node at $z$ from Propositions (2) and (3) of chapter (8) of \cite{seshadri1982}.
	When $C$ has a cusp at $z$, the statement follows from the main theorem of \cite{martinsetal-torsionfree2012}.

	Because of its structure, a deformation of $E_z$ may be given by merely deforming $\mathfrak{m}_z$. 
	We will exploit this feature of $E_z$ to produce a local deformation $(\mathcal C_z,\mathcal E_z)$ of $(C_z,E_z)$ over the disc which smooths $C$ at $z$ and deforms $E_z$ into a vector bundle.
	If $z$ is a node of $C$, then \cite[Lemma~9.2.2]{pandharipande1996} gives an explicit deformation of $\mathfrak{m}_z$, which we have seen is adequate.
	If $z$ is a cusp, then let $S$ be a neighborhood of $z$ isomorphic to $\Spec(\C[x,y]/(y^2-x^3))$.
	Let $\mu:\Spec(\C[x,y,t]/(y^2-x^3-2t^6)\rightarrow\Spec(\C[t])$ be the projection map.
	We claim that the ideal $\mathcal{I}:=(x-t^2,y-t^3)$ is the desired deformation of $\mathfrak{m}_z$.
	To see this, observe that $\mathcal{I}$ defines a section of $\mu$, whose image we will call $L$, satisfying an exact sequence
	\begin{equation}
		0\rightarrow\mathcal{I}\rightarrow\O_S\rightarrow\O_L\rightarrow 0.
		\label{eqn:irreducible-1}
	\end{equation}
	Because $\O_S$ is torsion-free over $\C[t]$, so is $\mathcal{I}$.
	Thus, $\mathcal{I}$ is $\mu$-flat because $\C[t]$ is a Dedekind domain.
	Moreover the section of $\mu$ is an isomorphism, and so $\O_L$ is $\mu$-flat. 
	Hence~\eqref{eqn:irreducible-1} is exact after restriction to $z$.
	Thus, $\mathcal{I}_0\cong m_{z}$ and so $\mathcal{I}$ is the desired deformation.

	\vskip .1in
	At this point, we have produced a local deformation $(\mathcal{C}_z,\mathcal{E}_z)$ of the germ $(C_z,E_z)$ over the disc which smooths $C_z$ and deforms $E_z$ to a vector bundle.
	Let $(\hat{\mathcal{C}}_z,\hat{\mathcal{E}}_z)$ be the associated formal deformation.
	The collection of these formal local deformations at each singular point defines a formal deformation for the deformation functor $\Def^{loc}(C,E):=\prod_z \Def(C_z,E_z)$, where $\Def(C_z,E_z)$ is the local deformation functor for the pair $(C_z,E_z)$.
	As established in \cite[Section A.]{fantechi1997euler}, the morphism
	\[\Def(C,E)\xrightarrow{loc} \Def^{loc}(C,E)\]
	is smooth, where $\Def(C,E)$ is the deformation functor of the pair $(C,E)$ and $loc$ is the natural restriction map.
	Because $loc$ is smooth, we may lift the formal local deformation to a global formal deformation $(\hat{\mathcal{C}},\hat{\mathcal{E}})$ of $(C,E)$.
	This global formal deformation is effective; the proof is identical to the standard proof for deformations of schemes, e.g., \cite[Thm.~2.5.13]{sernesi_defalgsch}.
	This effective  deformation is then algebraizable by a special case of Artin's algebraization theorem (e.g., \cite[Thm.~2.5.14]{sernesi_defalgsch}).	Let $(\mathcal C,\mathcal E)$ be an algebrized deformation.  Restriction to a disc gives the theorem.  
	
	Alternatively, a direct gluing argument may be carried out to explicitly construct a global deformation from the local deformation, as in \cite[Lemma 9.2.3]{pandharipande1996}.
\end{proof}
\begin{prop}	\label{prop:irreducible}
	$\overline{U}^{ps}_{e,r,g}$ is an irreducible variety.	
\end{prop}
\begin{proof}
	Consider $\pi_{SS}:Q^{SS}\rightarrow H_g$.
	By \cite[Prop. 24]{seshadri1982}, the scheme
	\[\pi^{-1}_{SS}([C])\]
	is irreducible for each nonsingular curve $C$, $[C]\in H_g$.
	Because the locus $H_g^0\subset H_g$ of nonsingular curves is irreducible, $\pi^{-1}_{SS}(H^0_g)$ is irreducible.
	By Lemma~\ref{lemma:irreducible-3}, $\pi^{-1}_{SS}(H^0_g)$ is dense in $Q^{SS}$.
	There is a surjection
	\[Q^{SS}\rightarrow \overline{U}^{ps}_{e,r,g},\]
	whence we conclude $\overline{U}^{ps}_{e,r,g}$ is irreducible.
\end{proof}

\bibliography{biblio}
\bibliographystyle{amsalpha}

\bigskip
\noindent
Department of Mathematics \\
Boston College\\
Chestnut Hill, MA 02467
USA

\bigskip
\noindent
{\tt matthew.grimes@bc.edu }
\end{document}